\documentclass[12pt,a4paper,leqno]{amsart}
\usepackage[utf8]{inputenc}
\usepackage[T1]{fontenc}
\usepackage[UKenglish]{babel}
\usepackage[a4paper,margin=1in]{geometry}
\usepackage[german=guillemets]{csquotes} 

\usepackage{graphicx} 
\usepackage{url}

\usepackage{natbib}
\setcitestyle{numbers,square}

\usepackage{amsmath,amsfonts,amssymb,amsthm}
\usepackage{eucal,mathrsfs,dsfont}
\usepackage[fleqn,tbtags]{mathtools}

\theoremstyle{plain}
\newtheorem{theorem}{Theorem}
\newtheorem{lemma}[theorem]{Lemma}

\newtheorem{corollary}[theorem]{Corollary}

\theoremstyle{definition}
\newtheorem{remark}[theorem]{Remark}
\newtheorem{definition}[theorem]{Definition}
\newtheorem{example}[theorem]{Example}

\numberwithin{theorem}{section}
\numberwithin{equation}{section}

\renewcommand{\leq}{\leqslant}

\renewcommand{\geq}{\geqslant}
\renewcommand{\ge}{\geqslant}

\DeclareMathOperator*{\supp}{supp}  
\DeclareMathOperator*{\sgn}{sgn}  
\DeclareMathOperator*{\spann}{span}  
\DeclareMathOperator*{\Leb}{Leb}  
\DeclareMathOperator*{\law}{law}  
\DeclareMathOperator*{\Ree}{Re}  
\DeclareMathOperator*{\Imm}{Im}

\newcommand{\nat}{\mathds{N}}
\newcommand{\integer}{\mathds{Z}}

\newcommand{\real}{\mathds{R}}
\newcommand{\rn}{{\mathds{R}^n}}
\newcommand{\complex}{\mathds{C}}

\newcommand{\Ee}{\mathds{E}}
\newcommand{\Pp}{\mathds{P}}

\newcommand{\Acal}{\mathcal{A}}

\newcommand{\Dcal}{\mathcal{D}}
\newcommand{\Fcal}{\mathcal{F}}
\newcommand{\Lcal}{\mathcal{L}}
\newcommand{\Mcal}{\mathcal{M}}
\newcommand{\Xcal}{\mathcal{X}}

\newcommand{\I}{\mathds{1}}

\begin{document}

\title[The Unique Continuation Principle for Nonlocal Operators]{\bfseries On the Unique Continuation Principle for a Class of Translation Invariant Nonlocal Operators}

\author[D.~Berger]{David Berger}
\address[D.~Berger]{
	TU Dresden, 
	Fakult\"at Mathematik, 
	Institut f\"ur Mathematische Stochastik, 
	01062 Dresden, Germany. 
	E-Mail: \textnormal{david.berger2@tu-dresden.de}}

\author[R.L.~Schilling]{Ren\'e L.\ Schilling}
\address[R.L.~Schilling]{
	TU Dresden, 
	Fakult\"at Mathematik, 
	Institut f\"ur Mathematische Stochastik, 
	01062 Dresden, Germany. 
	E-Mail: \textnormal{rene.schilling@tu-dresden.de}}




\thanks{Our research was supported by the Dresden--Leipzig \emph{ScaDS.AI} centre. We are grateful to Luz Roncal (Ikerbasque, Bilbao) for helpful discussions on this topic.}

\maketitle

\allowdisplaybreaks

\begin{abstract}
	The unique continuation property (UCP) for an operator $A$ says that, if $Au = 0 = u$ holds on an open set $G$, then one has $u=0$ everywhere. We establish necessary and sufficient conditions for the UCP for the class of L\'evy operators. We prove a connection between the UCP of the L\'evy operator and its resolvent. Our results are applied to obtain a new elementary proof of the UCP for the fractional Laplace operator, and for certain functions (Bernstein functions) of the discrete Laplace operator. 
\end{abstract}


Caffarelli and Silvestre observed in \cite{caf-sil} that the fractional Laplace operator $(-\Delta)^s u$, $0<s<1$, appears as a suitable limit of the Dirichlet-to-Neumann map of the following extension problem:
\begin{gather*}
	\mathrm{div}\left(y_{n+1}^{1-2s}\nabla U = 0\right)\quad\text{in $(0,\infty)\times\rn$ and}\quad U=u\quad\text{on $\rn$}.
\end{gather*}
Ever since, the  fractional Laplacian and related non-local operators have attracted considerable interest in analysis. Within probability theory, the fractional Laplacian and more general L\'evy operators (sometimes called \emph{Laplaciens generalis\'es}, see \cite{herz}) are well-known as infinitesimal generators of stochastic processes with independent and stationary increments (L\'evy processes) and convolution semigroups.

We are interested in the so-called \emph{unique continuation property} or \emph{principle} (UCP, for short). An operator $L$ acting on a class of functions defined on $\rn$ enjoys the UCP, if for some non-empty open set $G\subset\rn$
\begin{gather}\label{ucp}
	L u|_G = 0\;\; \&\;\; u|_G=0 \implies u=0 \quad\text{on $\rn$}.
\tag{\textsf{ucp}}
\end{gather} 
Obviously, if the UCP holds for some open ball $B_r(x)$, then it holds for all open sets $G$, which contain $B_r(x)$. If, in addition, $L$ is invariant under translations (as is the case here), the UCP holds for all open sets $G$ as soon as it holds for all open balls.

Building on earlier work on second-order elliptic partial differential equations and the Caffarelli--Silvestre representation of $(-\Delta)^s$, Fall and Felli \cite{fal-fel} established the UCP for the fractional Laplacian $(-\Delta)^s$, which may be perturbed by a (non-linear) zero-order term. As we will see, in its simplest form for $L=(-\Delta)^s$, this is equivalent to a result due to M.~Riesz~\cite[Ch.~VIII]{riesz}, see Remark~\ref{rs-03} below, as well as Ghosh \emph{et al.}~\cite[Rem.~4.2]{ghosh-et-al}. Following \cite{fal-fel}, variants of the UCP for $(-\Delta)^s$ were considered by many authors, for example by Rüland~\cite{ruland}, Ghosh and co-workers \cite{ghosh-et-al,ghosh-et-al-2} and, in a discrete setting, by Fern\'andez-Bertolin \emph{et al.} \cite{ber-et-al}. Let us point out that the paper \cite{fal-fel} and all subsequent papers allow for $Lu$ in \eqref{ucp} to be understood in a weak sense, while the potential-theoretic approach of Riesz requires a pointwise interpretation of $Lu$, hence more regularity of the function $u$. 

In this paper we obtain necessary and sufficient conditions, which guarantee that a L\'evy operator, and in particular the fractional Laplacian, satisfies the UCP. For this, it is more convenient to work with the classical representation of the fractional Laplacian as a principal-value integral
\begin{gather*}
	(-\Delta)^s u(x) = \lim_{\epsilon \to 0} c_s\int_{|y|>\epsilon} \left(u(x+y)-u(x)\right) |y|^{-n-2s}\,dy,\quad x\in\rn,
\end{gather*}
($c_s = s4^s\pi^{-n/2}\Gamma\left(s+\frac n2\right)/\Gamma(1-s)$, $s\in (0,1)$) for which there is an analogue for L\'evy operators, see \eqref{nlo-e04} below. In particular, we obtain a new elementary proof of the UCP for $L=(-\Delta)^s$, which differs from \cite{riesz} and \cite{fal-fel,ghosh-et-al}, see Corollary~\ref{ucpl-21}.

\section{A class of nonlocal operators}\label{nlo}

Our standard references for the material in this section are the monographs \cite{BSW} and \cite{jacob-i}.
\begin{definition}\label{nlo-03}
	A \emph{L\'evy operator} is an integro-differential operator $A$ defined on the test functions $C_c^\infty(\rn)$ by
	\begin{gather}\label{nlo-e04}\begin{aligned}
		Au(x)
		&:= \ell\nabla u(x) + \frac 12\textrm{div}\left(Q\nabla u(x)\right)\\
		&\qquad\mbox{}+ \int_{y\neq 0} \left(u(x+y)-u(x)-y\nabla u(x)\I_{(0,1)}(|y|)\right)\nu(dy),
	\end{aligned}\end{gather}
	where the \emph{characteristic triplet} $(\ell,Q,\nu)$ comprises $\ell\in\rn$, $Q\in\real^{n\times n}$ a positive semidefinite matrix, and the measure $\nu$ satisfying $\int_{y\neq 0} \min\left\{1,|y|^2\right\}\nu(dy)<\infty$.
\end{definition}
The fractional Laplacian $(-\Delta)^s$ is a L\'evy operator with the triplet $(0,0,c_s|y|^{-n- 2s})$. It is not hard to see that the formula \eqref{nlo-e04} extends to twice differentiable bounded functions $u\in C_b^2(\rn)$ or $u\in W^2_p(\rn)$, $p\in [1,\infty)$, see e.g.\ \cite{sch-ieot}. 

With a few elementary calculations we can represent $A$ as a pseudo-differential operator 
\begin{gather}\label{nlo-e06}
	Au(x) = -\psi(D)u(x) = \Fcal^{-1}\left(-\psi \Fcal u\right)(x)
	\qquad\text{($\Fcal$ denotes the Fourier transform)}
\end{gather}
with \emph{negative definite symbol} $\psi:\rn\to\complex$
\begin{gather}\label{nlo-e08}
	\psi(\xi)
	= -i\ell\xi + \frac 12Q\xi + \int_{y\neq 0}\left(1-e^{i\xi y}+i\xi y\I_{(0,1)}(|y|)\right)\nu(dy).
\end{gather}
The operator $(A,C_c^\infty(\rn))$ is closable in all spaces $L^p(\rn;dx)$, $1\leq p < \infty$, as well as in $C_\infty(\rn) := \left\{u\in C(\rn)\mid \lim_{|x|\to\infty} u(x)=0\right\}$ equipped with the sup-norm. We denote these closures as $(A,\Dcal(A))$ and specify in each case the underlying Banach space $\Xcal = L^p(\rn)$ or $\Xcal=C_\infty(\rn)$. In either case, $(A,\Dcal(A))$ is the infinitesimal generator of a strongly continuous contraction semigroup $(P_t)_{t\geq 0}$ of convolution operators on $\Xcal$
\begin{gather*}
	P_t u = \Fcal^{-1}\left(e^{-t\psi}\Fcal u\right),\;\; u\in C_c^\infty(\rn)
	\quad\text{or}\quad
	P_t u(x) = \int u(x+y)\,\mu_t(dy),\;\; u\in \Xcal;
\end{gather*}
the measures $(\mu_t)_{t\geq 0}$ are probability measures satisfying $\Fcal \mu_t(\xi) = e^{-t\psi(\xi)}$. 

From now on we assume, for simplicity and in order to stay compatible with the literature, that $\psi$ is real-valued\footnote{This assumption is needed only for the $L^p$-setting. With some effort one can relax this assuming the \emph{sector condition} $|\Imm \psi(\xi)|\leq c\Ree\psi(\xi)$ for some constant $c\geq 0$ and all $\xi\in\rn$.}. In the $L^p$-setting the scale of anisotropic $\psi$-Bessel potential spaces $H^{\psi,t}_p(\rn)$, $t\in\real$, plays for $-\psi(D)$ exactly the same role as the Sobolev scale $H^{t}_p(\rn)$, $t\in\real$, for $(-\Delta)^{s}$, see \cite[Chapter 3.10]{jacob-i} for $p=2$ and \cite[Section 2]{fjs-2} for $p\in [1,\infty)$. By definition, $u\in H^{\psi,t}_p(\rn)$ if, and only if, $u = (1-A)^{-t/2}f$ for some $f\in L^p(\rn)$ where $A$ is the closure of the operator $(-\psi(D),C_c^\infty(\rn))$ in $L^p(\rn)$. We can characterize the norm in $H^{\psi,t}_p(\rn)$ via
\begin{gather}\label{nlo-e10}
	\|u\mid H_p^{\psi,t}\| = \|\Fcal^{-1}[(1+\psi)^{t/2}\Fcal u] \mid L^p\|,\quad u\in C_c^\infty(\rn).
\end{gather}
If $\psi(\xi) \geq c |\xi|^\rho$ for some $\rho>0$ and $|\xi|\gg 1$, then the scales $H^{\psi,s}_p(\rn)$ and $H^{t}_p(\rn)$ are \enquote{interlaced}, see \cite[Theorem~3.1]{jac-sch-sinica}, i.e.\ we have the following continuous embeddings:
\begin{gather}\label{nlo-e12}
	H^{\psi,s}_p(\rn) \hookrightarrow H^t_p(\rn) \hookrightarrow H^{\psi,r}_p(\rn)
	\quad s > \frac{t+n}{\rho}\quad\text{and}\quad t > r+n.
\end{gather}
Finally, $C_\infty^k(\rn)$ denotes the functions which are $k$-times continuously differentiable and vanish, with all derivatives, at infinity.
\begin{lemma}\label{nlo-05}
	If the underlying Banach space $\Xcal$ is
	\begin{enumerate}\itemsep=5pt 
	\item[a)]  $\Xcal = C_\infty(\rn)$, then $C^\infty_\infty(\rn)\subset C_\infty^2(\rn)\subset\Dcal(A)$.
	\item[b)]  $\Xcal = L^p(\rn)$, $p\in [1,\infty)$, then $\bigcap_{t>0} H^t_p(\rn) \subset H^{\psi,2}_p(\rn) = \Dcal(A)$.
	\end{enumerate}
\end{lemma}
\begin{proof}
	a) Since $C_c^\infty(\rn)\subset \Dcal(A)$ and \eqref{nlo-e04} implies $\|Au\|_\infty \leq c \sum_{|\alpha|\leq 2}\|\partial^\alpha u\|_\infty$, it follows from the closedness of $(A,\Dcal(A))$ that $C_\infty^2(\rn)\subset \Dcal(A)$. The inclusion $C_\infty^\infty(\rn)\subset C_\infty^2(\rn)$ is trivial.

\medskip	
	b)  The equality $H^{\psi,2}_p(\rn) = \Dcal(A)$ is from \cite[Thm.~2.1.15]{fjs-2}, while $\bigcap_{t>0} H^{t}_p(\rn)\subset\Dcal(A)$ follows from the second embedding in \eqref{nlo-e12}.
\end{proof}

The measures $(\mu_t)_{t\geq 0}$ are the transition probabilities of a \emph{L\'evy process}, i.e.\ a stochastic process $(X_t)_{t\geq 0}$ with values in $\rn$, right-continuous trajectories, and stationary and independent increments,
\begin{gather*}
	\mu_t(B) = \Pp(X_t \in B),\quad \text{for all Borel sets\ \ } B\subset\rn.
\end{gather*}
Using methods from probabilistic potential theory, it is possible to obtain stochastic representations of the solution to the Dirichlet problem
\begin{gather}\label{nlo-e14}
	\left[-\psi(D)u\right]\big|_G = 0
	\quad\text{and}\quad
	u\big|_{G^c} = f\big|_{G^c},
\end{gather}
which is, in view of the non-locality of $\psi(D)$, an exterior-domain problem rather than a boundary value problem. In probabilistic notation we have
\begin{gather}\label{nlo-e16}
	u(x) = \Ee^x\left[f\left(X_{\tau_G}\right)\right] = \int_{G^c} f(y) \,\Pp^x\left(X_{\tau_G}\in dy\right),
	\quad \tau_G := \inf\left\{t>0\mid X_t \notin G\right\},
\end{gather}
i.e., $\tau_G$ is the first time a process exits the set $G$ and the probability distribution of $X_{\tau_G}$ is the harmonic measure. We refer to \cite{hoh-jac96} for a detailed discussion on the properties of the stochastic solution.

For the present paper the following intuitive understanding of \eqref{nlo-e16} is useful: By definition, $X_{\tau_G}\in G^c$ and it is known that the L\'evy process leaves the domain $G$ almost surely by a jump $\Delta X_{\tau_G}$ (we write $\Delta X_t := X_t - X_{t-}$ for a jump discontinuity of the trajectory). It is known that  the support of the L\'evy measure $\nu$ appearing in \eqref{nlo-e04} describes the shape and size of $\Delta X_t$, and the mass distribution of the L\'evy measure governs the frequency of $\Delta X_t$. This means that the support of $\nu$ influences the support of the harmonic measure $\law\left(X_{\tau_G}\right)$.

If we compare \eqref{ucp} with \eqref{nlo-e14} and \eqref{nlo-e16}, we would expect that the UCP fails if $X_{\tau_G}$ cannot reach every open set in $G^c$ with strictly positive probability (this is indeed true, as we will see in Example~\ref{ucpl-11}), but the converse fails spectacularly: even if $\nu$ has full support and $X_{\tau_G}$ hits every open set in $G^c$ with strictly positive probability it may happen that the UCP fails, see Examples~\ref{ucpl-13}--\ref{ucpl-17}. 

\section{The unique continuation principle for L\'evy operators}\label{ucpl}

Let $A$ be a L\'evy operator of the form \eqref{nlo-e04}, with characteristic triplet $(\ell,Q,\nu)$, L\'evy process $(X_t)_{t\geq 0}$ and characteristic exponent $\Ee e^{i\xi X_t}=\exp(-t\psi(\xi))$, where $\psi$ is given by \eqref{nlo-e08}. We write $(A,\Dcal(A))$ for the closure of the operator $(A,C_c^\infty(\rn))$. 
\begin{definition}\label{ucpl-03}
	We say that a L\'evy operator $(A,\Dcal(A))$ satisfies the unique continuation principle, if the condition \eqref{ucp} holds for some non-empty open set $G\subset\rn$ and every $u\in D(A)$.
\end{definition}
This property was already discussed for the fractional Laplacian, which is the generator of a stable L\'evy process, but here we will discuss it in a wider context. We need a simple lemma from functional analysis to present our first result. 

Recall the definition of annihilators for Banach spaces in duality. Let $\Xcal$ be a Banach space and $\Xcal^*$ its topological dual. Let $M$ be a subspace of $\Xcal$ and $N$ a subspace of $\Xcal^*$. We denote by
\begin{align*}
	M^{\perp}
	&:=\left\{x^*\in \Xcal^*\mid \forall x\in M\::\: \langle x^*,x\rangle=0\right\},\\
	\prescript{\perp}{}{N}
	&:=\left\{x\in \Xcal\mid\forall x^*\in N\::\:\langle x^*,x\rangle=0\right\},
\end{align*}
the \emph{annihilators} of $M$ and $N$, see e.g.\ \cite[Chapter 4]{RudinFunc}. Let us also record a standard consequence of the Hahn-Banach theorem.
\begin{lemma}\label{ucpl-05}
	Let $\Xcal$ be a Banach space. A subspace $\Lcal\subset \Xcal^*$ is dense in $\Xcal^*$ in the weak-$\ast$-topology if, and only if, $\prescript{\perp}{}{\Lcal}=\{0\}$.
\end{lemma}
Fix an open set $G\subset\rn$. A sequence of Radon measures $\mu_n$ satisfying $\sup_n \mu_n(G)<\infty$ converges vaguely to a (necessarily finite) Radon measure $\mu$ on $G$, if for all $\phi\in C_\infty(G)$ it holds that
\begin{align*}
	\lim_{n\to\infty}\int_G \phi(x)\mu_n(dx) = \int_G \phi(x)\mu(dx).
\end{align*}
This is exactly the weak-$\ast$-convergence for the dual pair $(C_\infty(G),C_\infty(G)^{\ast})$ (recall that $C_\infty(G) = \overline{C_c(G)}^{\|\bullet\|_\infty}$).

\begin{theorem}\label{ucpl-07}
Let $(A,\Dcal(A))$ be a L\'evy operator defined on $C_\infty(\rn)$ and denote by $(\ell,Q,\nu)$ its characteristic triplet. The operator $A$ satisfies the unique continuation principle if, and only if, for every $\epsilon>0$ the linear span of the set
\begin{gather*} 
	\Sigma(\nu,\epsilon):=\left\{\nu(\cdot+x)\big|_{\rn\setminus \overline{B_\epsilon(0)}} \;\middle|\; x\in B_\epsilon(0)\right\}
\end{gather*} 
is dense with respect to vague convergence in the space $\Mcal_b\left(\rn\setminus\overline{B_\epsilon(0)}\right)$ of bounded signed Radon measures on $\rn\setminus\overline{B_\epsilon(0)}$.
\end{theorem}
\begin{proof}
Let $\phi\in\Dcal(A)$ be such that both $A\phi(x)=0$ and $\phi(x)=0$ for $x\in B_{\epsilon}(0)$ for some $\epsilon>0$. Without loss of generality we may assume that $\phi\in C^\infty(\rn)$. Otherwise, we pick an $\chi\in C^\infty_c(\rn)$ such that $\supp(\chi)\subset B_{\epsilon/2}(0)$. Clearly, $\chi\ast \phi \in C_\infty^\infty(\rn)\subset\Dcal(A)$. Moreover, $A(\chi\ast\phi) = \chi\ast A\phi=0$ in $B_{\epsilon/2}(0)$. Since $\epsilon>0$ is arbitrary, we can replace $\phi$ by $\chi\ast\phi$. 

For $\phi\in C^\infty(\rn)\cap\Dcal(A)$ and $x\in B_\epsilon(0)$ we have
\begin{align}\label{ucpl-e04}\begin{split}
	A\phi(x)
	&= \ell \nabla \phi(x) +\frac{1}{2}\textrm{div}\left( Q\nabla \phi(x)\right)\\
	&\qquad\mbox{}+\int_{y\neq 0} \left(\phi(x+y)-\phi(x)-y\nabla \phi(x)\I_{(0,1)}(|y|)\right)\nu(dy)\\
	&= \int_{y\neq 0} \phi(x+y)\,\nu(dy)
	= \int_{\rn\setminus \overline{B_{\epsilon}(0)}} \phi(y)\,\nu(dy-x).
\end{split}
\end{align}
This shows that $\phi\in\Acal$, where the set $\Acal$ is defined as
\begin{gather*}
	\Acal := 
	\left\{f\in C_\infty\left(\rn\setminus \overline{B_{\epsilon}(0)}\right) \;\middle|\; \forall x\in B_{\epsilon}(0)\::\: \int_{\rn\setminus \overline{B_{\epsilon}(0)}} f(y)\,\nu(dy-x) = 0\right\}
\end{gather*}
We claim that $\Acal=\{0\}$ if, and only if, $\phi\equiv 0$. 

Indeed, if $\Acal=\{0\}$ we get from $\Acal=\prescript{\perp}{}{\spann\left( \Sigma(\nu,\epsilon)\right)}$ and Lemma \ref{ucpl-05} that $\spann\left( \Sigma(\nu,\epsilon)\right)$ is dense, hence $\phi\equiv 0$. 

Conversely, assume that there exists some $\phi\in\prescript{\perp}{}{\spann\left( \Sigma(\nu,\epsilon)\right)}\setminus\{0\}$. We can extend $\phi$ continuously onto $\rn$ if we set $\phi(x):=0$ for all $x\in\overline{B_\epsilon(0)}$. Using a mollifier argument as at the beginning of the proof we can assume that $\phi\in C_\infty^\infty(\rn)$, hence $\phi\in \Dcal(A)$, and we get from \eqref{ucpl-e04}
\begin{align*}
	A\phi(x) = 
	\int_{\rn\setminus\{0\}} \phi(x+y)\,\nu(dy)
	= \int_{\rn\setminus \overline{B_{\epsilon}(0)}} \phi(y)\nu(dy-x)
	= 0
\end{align*}
for all $x\in B_\epsilon(0)$. Therefore, $A$ cannot satisfy the unique continuation principle.
\end{proof}

Note that the differential operator part of $A$ has no influence on the validity of the UCP when there is a non-trivial non-local part, i.e.\ a non-trivial $\nu$.

The following lemma shows that the validity of the UCP in $C_\infty(\rn)$ for a given L\'evy operator $A$ guarantees its validity in an $L^p$-context, too. Further, the lemma shows that in $L^p$ the implication \eqref{ucp} may be applied for any $u\in H^{\psi,t}_p$ and any $t\in\real$ -- in this case $Au$ is understood in a weak sense. In order to keep things apart, we write for a moment $(A_p,\Dcal(A_p))$ and $(A_\infty,\Dcal(A_\infty))$ for the closure of $(A,C_c(\rn))$ in $L^p(\rn)$ and $C_\infty(\rn)$, respectively.
 
As the symbol $\psi$ is not necessarily smooth, it makes sometimes sense to look at the symbol 
\begin{gather*}
	\psi_1(\xi) := \psi(\xi) - \int_{|y|\geq 1} \left(1-e^{i\xi y}\right) \nu(dy), \quad\xi\in\rn.
\end{gather*}
The symbol $\psi_1$ is smooth and all derivatives are polynomially bounded, see \cite[Theorem 3.7.13]{jacob-i}. Furthermore, it holds that $H^{\psi,s}_p(\rn)=H^{\psi_1,s}_p(\rn)$, see \cite[Theorem 3.10.4]{jacob-i}. Assuming that $u\in H^{\psi,s}_p(\rn)\subset \mathcal{S}'(\rn)$, we know that there exists an $f\in L^p(\rn)$ such that $u=(1-A_1)^{-s/2}f$, where $A_1|_{C_c^\infty(\rn)}=-\psi_1(D)|_{C_c^\infty(\rn)}$. It is then easy to see that the convolution of $u$ and a function $\phi\in C^\infty_c(\mathbb{R}^n)$ is given by
\begin{align*}
	u\ast \phi(x)=f\ast ((1-A_1)^{-s/2})^*\phi(x).
\end{align*} 
We know that the Fourier transform of $((1-A_1)^{-s/2})^*\phi$ is given by $(1+\psi_1(-\xi))^{-s/2}\mathcal{F}\phi$, which implies that $((1-A_1)^{-s/2})^*\phi\in L^q(\rn)$ for all $q\ge 1$. Hence $f\ast ((1-A_1)^{-s/2})^*\phi(x)\in C^\infty_\infty(\rn)$ as $((1-A_1)^{-s/2})^*\phi \in L^{\frac{p}{p-1}}(\rn)$.
\begin{lemma}\label{ucpl-09}
	Let $(A,\Dcal(A))$ be a L\'evy operator with characteristic triplet $(\ell,Q,\nu)$ and characteristic exponent $\psi(\xi)$. If $(A_\infty,\Dcal(A_\infty))$ satisfies the UCP in the sense of Definition~\ref{ucpl-03}, then $(A_p,\Dcal(A_p))$, $p\in [1,\infty)$, has the following property: For some open set $G = B_{2\epsilon}(x_0)$ and for $u\in H^{\psi,t}_p(\rn)$ where $s\in\real$ is arbitrary it holds that
	\begin{gather}\label{ucpl-e06}
		(A_p u)|_G = 0\;\;\&\;\; u|_G = 0 \implies u = 0.
	\end{gather}	
\end{lemma}
\begin{proof}
	Assume that for some $p\in [1,\infty)$ and $t\in\real$ the function $u\in H^{\psi,t}_p(\rn)$ satisfies $(A_p u)|_G = 0 = u|_G$ for some $G=B_{2\epsilon}(x_0)$. The Friedrichs mollifier is defined as $u_\epsilon := J_\epsilon u := \chi_\epsilon\ast u$ with $\chi_\epsilon(x) := \epsilon^{-n}\chi(\epsilon^{-1}x)$ and $\chi\in C_c(\infty)$, $\chi\geq 0$, $\supp\chi\subset \overline{B_1(0)}$, $\int \chi(x)\,dx = 1$. By the discussion before we have seen that $J_\epsilon u\in C^\infty_\infty(\rn)$.
	Moreover, since $\supp\chi_\epsilon\subset \overline{B_\epsilon(0)}$, the convolution structure of the operator $J_\epsilon$ shows that in the smaller ball $B_\epsilon(x_0)$
	\begin{gather*}
		u_\epsilon|_{B_\epsilon(x_0)} = (J_\epsilon u)|_{B_\epsilon(x_0)} = 0
		\quad\text{and}\quad
		(A_p u_\epsilon)|_{B_\epsilon(x_0)} = (J_\epsilon \underbracket[.6pt]{A_p u}_{\mathclap{=0\text{\ on\ } B_{2\epsilon}(x_0)}})|_{B_\epsilon(x_0)} = 0.
	\end{gather*}
	Since $u_\epsilon\in \Dcal(A_\infty)\cap\Dcal(A_p)$, we have $A_p u_\epsilon = A_\infty u_\epsilon$, and the UCP for $(A_\infty,\Dcal(A_\infty))$ shows hat $u_\epsilon = 0$. This shows that $u=0$ (Lebesgue) almost everywhere, and the claim follows.
\end{proof}

Lemma~\ref{ucpl-09} allows us to consider only $(A_\infty,\Dcal(A_\infty))$, which we denote again as $(A,\Dcal(A))$. The condition in Theorem \ref{ucpl-07} is a bit cumbersome to check. But it is very useful if one wants to construct examples for which the UCP fails.
\begin{example}\label{ucpl-11}
	Let $A$ be a L\'evy operator whose L\'evy measure $\nu$ does not have full support. Then $A$ does not satisfy the UCP.
	
	Since $\supp\nu\neq\rn\setminus\{0\}$, then there is some ball $B_\delta(x_0)\subset\rn\setminus\{0\}$ such that $\supp\nu \cap B_\delta(x_0)=\emptyset$; in particular, $\nu|_{B_\delta(x_0)} = 0$. We want to show that the span of $\Sigma(\nu,\epsilon)$ cannot be dense for $\epsilon := \frac 12\delta$. Since $\nu|_{B_\delta(x_0)}=0$ the shifted measures satisfy $\nu(\cdot+x')|_{B_\eta(x_0)}=\nu|_{B_\eta(x_0+x')}=0$, $|x'|<\frac 12\delta$ and $\eta < \frac 12\delta$, see Fig.~\ref{fig-1}. This means that the \enquote{hole} in the support is bequeathed to all shifts. In particular, $\spann(\Sigma(\nu,\delta))$ cannot be dense in the bounded signed Radon measures $\Mcal_b\left(\rn\setminus\overline{B_\delta(0)}\right)$.
	\begin{figure}[!h]\centering
		\includegraphics[width = .6\textwidth]{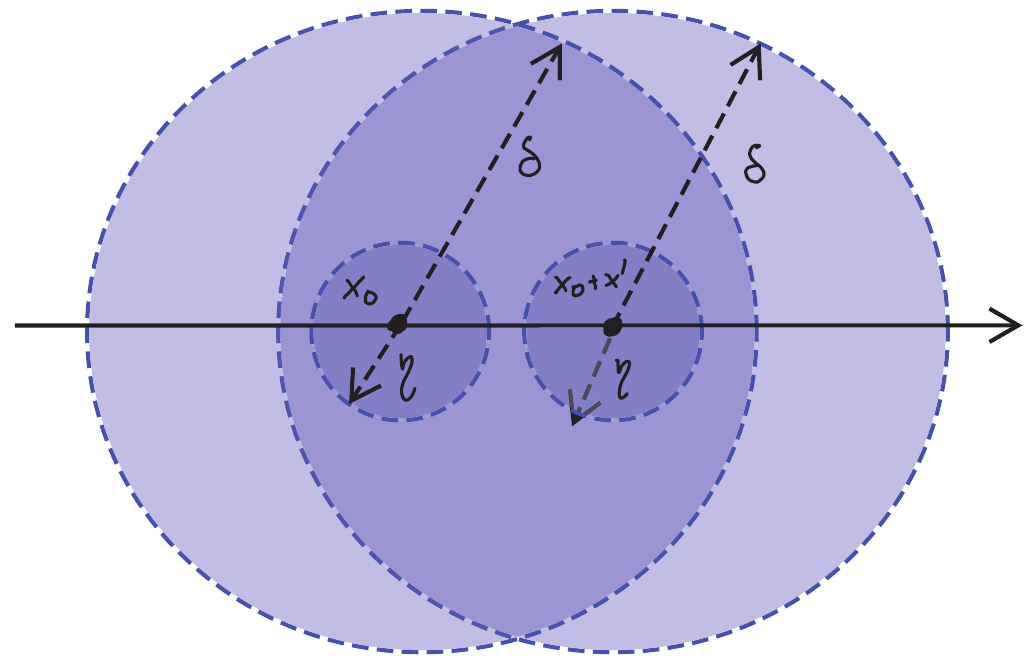}
		\caption{The intersection of $B_\delta(x_0)$ with the shifted ball $x'+B_\delta(x_0) = B_\delta(x_0+x')$, $|x'|<\frac 12\delta$, contains the balls $B_\eta(x_0)$ and $B_\eta(x_0+x')$ for any $\eta\in \left(0,\frac 12\delta\right)$. }\label{fig-1}
	\end{figure}
\end{example}

In particular, if we modify the L\'evy measure $\nu_s(dy) := c_s|y|^{-s/2-n}\,dy$ of the fractional Laplacian locally by \enquote{cutting a hole} into the support of $\nu_s$, the resulting operator will not enjoy the UCP. The \enquote{surgery} argument, which we used in Example~\ref{ucpl-11}, can be modified to other modifications of the L\'evy measure. The following example show that it is not sufficient to require full topological support of the L\'evy measure in order to get the UCP.
\begin{example}\label{ucpl-13}
	Let $A$ be a L\'evy operator whose L\'evy measure $\nu$ is such that $\nu|_{B_{\delta}(x_0)}=\Leb|_{B_{\delta}(x_0)}$ for some $x_0\in \rn\setminus\{0\}$ ($\Leb$ stands for Lebesgue measure). Then $A$ does not enjoy the UCP.
	
	Using the argument of Example~\ref{ucpl-13}, we see that the shifted L\'evy measures satisfy $\nu(\cdot+x')|_{B_{\eta}(x_0)}=\nu|_{B_{\eta}(x_0+x')}=\Leb|_{B_{\eta}(x_0+x')}$, hence $\spann(\Sigma(\nu,\delta))$ cannot be dense in $\Mcal_b\left(\rn\setminus\overline{B_\delta(0)}\right)$.
\end{example}

\begin{example}\label{ucpl-15}
	Let $A$ be a L\'evy operator whose L\'evy measure $\nu$ is such that  $\nu|_{B_{\delta}(x_0)}(dx) = e^{\beta x}\I_{B_{\delta}(x_0)}\,dx$ for some $\beta\in\rn$, and $x_0\in \rn\setminus\{0\}$, $\delta>0$. Then $A$ does not enjoy the UCP.
		
	If we take a measure $\mu\in \spann(\Sigma(\nu,\delta))$, then $\mu|_{B_{\eta}(x_0)}(dx) = Ce^{\beta x}\I_{B_{\eta}(x_0)}(x)\,dx$, where $C\in \real$ and $\eta = \frac 12\delta$. Thus, $\spann(\Sigma(\nu,\delta))$ cannot be dense in $\Mcal_b\left(\rn\setminus\overline{B_\delta(0)}\right)$.
\end{example}

\begin{example}\label{ucpl-17}
	Let $A$ be a L\'evy operator whose L\'evy measure $\nu$ is such that $\nu|_{B_{\delta}(x_0)}(dx) = p(x)\I_{B_{\delta}(x_0)}(x)\,dx$ for a polynomial $p(x)=\sum_{|\alpha|\leq m} c_\alpha x^\alpha$ (in multiindex notation), and some $x_0\in \rn\setminus\{0\}$, $\delta>0$. Then $A$ does not satisfy the UCP.
	
	This follows as in Example~\ref{ucpl-15}: all shifted measures $\nu(\cdot+x')$ have polynomial densities on the set $B_\eta(x_0)$. Since these are shiftes versions of $p(x)$, they have all the degree $\deg p = m$. Therefore, $\spann(\Sigma(\nu,\delta))$ cannot be dense in $\Mcal_b\left(\rn\setminus\overline{B_\delta(0)}\right)$.
\end{example}

We are now going to show how one can deduce the UCP for $(-\Delta)^s$ (or, equivalently, Riesz' theorem, see \eqref{riesz} at the beginning of Section~\ref{rs}) from Theorem~\ref{ucpl-07} in a rather elementary way. Our proof differs from the proofs of Riesz and Ghosh \emph{et al.}
\begin{corollary}\label{ucpl-21}
	The fractional Laplacian $\left((-\Delta)^{s}, \Dcal\left((-\Delta)^{s}\right)\right)$, $s\in (0,1)$, on $C_\infty(\rn)$ satisfies the UCP.
\end{corollary}
\begin{proof}
We consider first the one-dimensional case $n=1$. The L\'evy measure is given by $\nu(dy) = c |y|^{-1-2s}\,dy$, $s\in (0,1)$, with a normalization constant $c=c_s$. Fix $\epsilon>0$. Since $\Sigma(\nu,\epsilon)$ contains (small) shifts of $\nu(\cdot + x)$, we know that the closure of $\spann(\Sigma(\nu,\epsilon))$ contains the (signed) measures 
\begin{gather*}
	\nu^{(k)}(dy)
	:= \frac{d^k}{dy^k} |y|^{-1-2s}\,dy
	= c_{k,s}\cdot \sgn(y)^{k}|y|^{-k}|y|^{-1-2s}\,dy,\quad k\in\nat_0,
\end{gather*} 
It is easy to see that the span of $\left\{ \sgn(y)^k|y|^{-k} \mid k\in \nat\right\}$ is a subalgebra of $C_\infty\left(B^c_\epsilon(0)\right)$. This subalgebra separates points and vanishes nowhere, so we can use the Stone-Weier\-strass theorem to see that the closure of the subalgebra with respect to $\|\cdot\|_\infty$ is $C_\infty\left(B^c_\epsilon(0)\right)$. As $\left\{\phi(x)\,dy \mid \phi\in C_c\left(B^c_\epsilon(0)\right),\;\phi\geq 0\right\}$ is weakly dense in the space of bounded measures with support in $B_\epsilon^c(0)$, it is enough to show that we can approximate every $\phi\in C_c\left(B^c_\epsilon(0)\right)$ in the $L^1(B^c_\epsilon(0))$-norm by finite linear combinations of the measures $\nu^{(k)}$.

To do so, pick $\phi\in C_c(B^c_\epsilon(0))$ and a sequence $(\phi_n)_{n\in\nat}\subset \spann\left\{ \sgn(y)^{k}|y|^{-k} \mid k\in\nat\right\}$, which converges uniformly to $|y|^{1+2s}\phi(y)$. We see that
\begin{align*}
	\int_{B_\epsilon(0)^c}\left|\phi(y)-|y|^{-1-2s}\phi_n(y)\right|dy
	\leq \int_{B_\epsilon(0)^c} |y|^{-1-2s}\,dy \sup_{y\in B_\epsilon(0)^c}\left||y|^{1+2s}\phi(y)-\phi_n(y)\right|
	\xrightarrow[n\to\infty]{} 0.
\end{align*}
By construction, $|y|^{-1-2s}\phi_n(y)\,dy$ is in the span of the $\nu^{(k)}$'s, and so we have proved that all L\'evy operators with characteristic triplet $(\ell,Q,\nu),\; \ell,Q\in\real$, and, in particular (take $\ell=Q=0$), the one-dimensional fractional Laplacian, enjoy the UCP.

\bigskip 
We will now turn to the multivariate setting $n>1$. We are going to show that
\begin{align}\label{ucpl-e12} 
	\left\{\frac{x^{\beta}}{|x|^{2m+n+2s}} \;\middle|\; m\in\nat_0,\;\beta\in \nat_0^n,\; |\beta|\leq m\right\}
	\subset \overline{\spann\Sigma(\nu,\epsilon)}
	=: D,
\end{align}
where we use the standard multiindex notation $x^{\beta}:=x_1^{\beta_1}\cdots x_n^{\beta_n}$. We prove \eqref{ucpl-e12} by induction in $m\in\nat_0$. It is clear that \eqref{ucpl-e12} holds for $m=0$. Assume that \eqref{ucpl-e12} holds for some $m\in\nat_0$. For the induction step $m\to m+1$ we use that for any $i=1,\dots, n$
\begin{align*}
	\partial_{x_i} \frac{x^{\beta}}{|x|^{2m}}\frac{1}{|x|^{n+2s}}
	=-(2m+n+2s)x_i\frac{x^{\beta}}{|x|^{2(m+1)}}\frac{1}{|x|^{n+2s}}+(\partial_{x_i}x^{\beta})\frac{1}{|x|^{2m}}\frac{1}{|x|^{n+2s}}.
\end{align*} 
If we rearrange this identity, the induction assumption shows that $\frac{x^{\beta}}{|x|^{2(m+1)}}\frac{1}{|x|^{n+2s}}\in D$ for all $\beta\in \nat_0^n$ such that $m+1\geq |\beta|\geq 1$. For $\beta=0$ we observe that
\begin{align*}
	\partial_{x_i}^2\frac{1}{|x|^{2m+n+2s}}
	=(2m+n+2s)(2(m+1)+n+2s)\frac{x_i^2}{|x|^{2(m+2)+n+2s}}-\frac{(2m+n+2s)}{|x|^{2m+1+n+2s}}.
\end{align*}
Summing this identity over $i=1,\dots,n$ and rearranging the result shows that the function ${|x|^{-2(m+1)-n-2s}}$ is contained in $D$. This completes the induction step. 

From this point onward, we can use the argument of the one-dimensional case.
\end{proof}

\begin{remark}\label{ucpl-23}
	The strategy used in the proof of Corollary~\ref{ucpl-21} extends to one-dimensional L\'evy operators whose L\'evy measure is of the form
	\begin{align*}
		\nu(dy) = c_r |y|^{-1-2r}\I_{(-\infty,0)}(y)\,dy + c_s |y|^{-1-2s}(y)\I_{(0,\infty)}(y)\,dy
	\end{align*}
	for $r,s\in (0,1)$, $r\neq s$: such non-symmetric operators enjoy the UCP, too.
\end{remark}

\section{The unique continuation principle for resolvents and semigroups}\label{rs}

The UCP for the fractional Laplace operator is equivalent to a classical result from potential theory due to M.\ Riesz \cite[Chapter VIII]{riesz}. Denote by $R_\lambda := (\lambda + (-\Delta)^s)^{-1}$, $\lambda>0$ the resolvent operators of the fractional Laplacian and $\emptyset \neq G\subset\rn$ an open set. Then
\begin{gather}\label{riesz}\tag{\textsf{Riesz}}
	f|_G = 0\;\;\&\;\; R_1f|_G = 0
	\implies R_1f=0 \;\;\text{and}\;\; f=0.
\end{gather}
Set $L=(-\Delta)^s$ and $f=u-Lu$. Clearly,
\begin{gather*}
	Lu|_G = 0 = u|_G
	\iff
	f|_G = 0 = R_1f|_G
\end{gather*}
and this shows that the conclusions of the \eqref{ucp}, $u=0$, and \eqref{riesz}, $f=0$, are equivalent.

In fact, the property \eqref{riesz} is nothing but the UCP for the $1$-resolvent.

This argument extends literally to L\'evy operators and their $\lambda$-resolvent operators $R_\lambda = (\lambda+\psi(D))^{-1}$ for any $\lambda>0$, and on any Banach space $L^p(\rn)$, $p\in [1,\infty)$ or $C_\infty(\rn)$. We have
\begin{lemma}\label{rs-03}
	A L\'evy operator $(A,\Dcal(A))$ satisfies the UCP if, and only if, the $\lambda$-resolvent operator  $R_\lambda = (\lambda-A)^{-1}$ satisfies the UCP for some \textup{(}and then for all\textup{)} $\lambda>0$.
\end{lemma} 

We will now give a probabilistic interpretation of the UCP for a L\'evy operator. As before, we denote by $(X_t)_{t\geq 0}$ the L\'evy process, which is generated by the L\'evy operator $(A,\Dcal(A))$, $\Dcal(A)\subset C_\infty(\rn)$; by $\tau$ we denote an independent exponentially distributed random variable, i.e.\ $\law(\tau) = \lambda e^{-\lambda x}\,dx$, $\lambda, x>0$.

\begin{corollary}\label{rs-05}
	Let $(A,\Dcal(A))$, $\Dcal(A)\subset C_\infty(\rn)$, be a L\'evy operator and $(X_t)_{t\geq 0}$ the associated L\'evy process. Then $A$ has the unique continuation property if, and only if, for every independent exponentially distributed random variable $\tau$ and for every $\epsilon>0$, the span of $\Sigma(\law(X_\tau),\epsilon)$ is vaguely dense in the space of finite signed measures on $\rn\setminus \overline{B_{\epsilon}(0)}$.
\end{corollary}
\begin{proof}
	Let $\tau$ be an independent (of $(X_t)_{t\geq 0}$) exponentially distributed random variable and observe that $P_tu(x) = \Ee u(X_t+x)$ is the transition semigroup of the L\'evy process. Denote by $\lambda^{-1}\in (0,\infty)$ the mean value of $\tau$, i.e.\ $\law(\tau) = \lambda e^{-\lambda t}\,dt$, $t>0$. Then, 
	\begin{gather}\label{rs-e06}
		\Ee u(X_{\tau}+x) 
		= \int_0^\infty \Ee u(X_t+x) \lambda e^{-\lambda t}\,dt
		= \lambda \int_0^\infty P_tu(x)\,e^{-\lambda t}\,dt
		= \lambda R_{\lambda}u(x).
	\end{gather} 
	Therefore, the result follows from Lemma \ref{ucpl-05} and \ref{rs-03}.
\end{proof}

\begin{corollary}\label{rs-07}
	Let $(A,\Dcal(A))$, $\Dcal(A)\subset C_\infty(\rn)$, be a L\'evy operator and $(X_t)_{t\geq 0}$ the associated L\'evy process. If $A$ satisfies the UCP, then for every $\epsilon>0$ and for every independent exponentially distributed random variable $\tau$ the span of the measures \begin{gather*}
		\left\{\law(X_\tau)(x+\cdot)|_{\rn\setminus \overline{B_{\epsilon}(0)}}\;\middle|\; x\in B_{\epsilon}(0),\;t> 0\right\}
	\end{gather*} 
	is vaguely dense in the space of finite signed measures on $\rn\setminus \overline{B_{\epsilon}(0)}$.
\end{corollary}
\begin{proof}
	Assume that $\left(\law(X_\tau)(x+\cdot)|_{\rn\setminus \overline{B_{\epsilon}(0)}}\right)_{x\in B_{\epsilon}(0),t> 0}$ is not dense. Then, there exists (by Lemma \ref{ucpl-05} and smoothing out) a function $u\in C^\infty_\infty(\rn)\setminus\{0\}$ such that $u=0$ in and $P_tu=0$ in $B_{\epsilon}(0)$ for some $\epsilon>0$. We conclude that $R_{\lambda} u(x)=\int_0^\infty e^{-\lambda t} P_t u(x)\,dt=0$ for all $x\in B_{\epsilon}(0)$. In view of  Lemma \ref{rs-03}, the operator $A$ cannot satisfy the unique continuation principle.
\end{proof}

We conclude this section with an example which shows that there are L\'evy operators, which do do not satisfy the UCP (and whose resolvents do not satisfy the UCP, cf.\ Lemma~\ref{rs-03}), although the semigroups generated by them do have the UCP. 
\begin{example}
	The semigroup generated by Brownian motion satisfies the unique continuation principle. Since the generator of Brownian motion is $\frac 12\Delta$, it is obvious that neither the generator nor its resolvent satisfy the UCP.
	
	Denote by $p_t(x)=(2\pi t)^{-n/2}\exp(-|x|^2/t)$ the density of Brownian motion at time $t>0$. If $u\in C_\infty(\rn)$, the function $p_t\ast u$ can be holomorphically extended to $\mathbb{C}^n$. If $p_t\ast u=0$ in an open set on $\rn$, it follows that $p_t\ast u=0$ everywhere. Hence, it follows that that the semigroup generated by the Brownian motion satisfies the unique continuation principle.
\end{example}

\section{The uniform continuation principle for discrete L\'evy operators}\label{dis}

A non-degenerate \emph{random walk} on a lattice $L\subset\rn$, $0\in L$, is a stochastic process of the form
\begin{gather*}
	R_0:=0
	\quad\text{and}\quad 
	R_m := \sigma_1+\dots+\sigma_m,\quad m\in\nat,
\end{gather*}
where the steps $\sigma_k\neq 0$, $k\in\nat$, are independent, identically distributed random variables such that $p_\ell := \Pp(\sigma_1 = \ell)\geq 0$ and $\sum_{\ell\in L} p_\ell=1$. If we replace the number of steps $m\in\nat$ by a(n integer-valued) Poisson process $(N_t)_{t\geq 0}$,\footnote{The law of a Poisson random variable $N_t$ is $\Pp(N_t = m) = \frac{t^m}{m!} e^{-t}$ for $m\in\nat_0$.} which is independent of the steps $(\sigma_k)_{k\in\nat}$, then $R_{N_t}$ is a L\'evy process, and the characteristic exponent is given by
\begin{align*}
	\Ee e^{iR_{N_t}\xi}
	= \sum_{m=0}^\infty \Ee e^{i(\sigma_1+\dots+\sigma_m) \xi}\Pp(N_t = m)
	= \sum_{m=0}^\infty \left(\Ee e^{i\sigma_1 \xi}\right)^m \frac{t^m e^{-t}}{m!}
	= \exp\left[-t\left(1-\Ee e^{i\sigma_1\xi}\right)\right].
\end{align*}
Therefore,
\begin{gather}\label{dis-e01}
	\psi_R(\xi) = \sum_{\ell\in L} p_\ell\left(1-e^{i \ell\xi}\right),\quad \xi\in\rn,
\end{gather}
i.e.\ $\psi_R$ is of the form \eqref{nlo-e08} with the triplet $(0,0,\nu)$ and $\nu(dy) = \sum_{\ell\in L} p_\ell\delta_\ell(dy)$. The infinitesimal generator is
\begin{gather}\label{dis-e02}
	A_Ru(x) = -\psi_R(D)u(x) = \sum_{\ell\in L} p_\ell\left(u(x+\ell)-u(x)\right),\quad x\in\rn.
\end{gather}

If $R$ is the nearest-neighbour random walk on $L=\integer^n$, i.e.\ $\Pp(\sigma_1 = \pm e_j) = (2n)^{-1}$ for the coordinate unit vectors $e_j\in\integer^n$,  $j=1,\dots, n$, then \eqref{dis-e02} becomes the discrete Laplacian
\begin{gather}\label{dis-e04}
\begin{aligned}
	\Delta_n u(x) 
	= \frac 1{2n}\sum_{|e|=1} \left(u(x+e)-u(x)\right)
	= \frac 1{2n}\sum_{k=1}^n \left(u(x+e_k)+u(x-e_k)-2u(x)\right).
\end{aligned}
\end{gather}

A simple way to construct the fractional discrete Laplacian is to take the $s$-th root of $\Delta_n$ by using a suitable functional calculus, see e.g.\ \cite{ber-et-al} for an approach via Fourier series. We use the \emph{Bernstein functional calculus}, see \cite[Chapter 13]{SSV}, and \emph{Bochner's subordination} principle. This allows us to consider functions $f(-A_R)$ of the infinitesimal generator $A_R$ (e.g.\ $A_R=\Delta_n$), where $f:(0,\infty)\to [0,\infty)$ is a Bernstein function, i.e.\ a $C^\infty$-function given by the following formula
\begin{gather}\label{dis-e06}
	f(\lambda) = \beta\lambda + \int_{(0,\infty)}\left(1-e^{-\lambda t}\right) \rho(dt),\quad \lambda >0,
\end{gather}
with $\beta\geq 0$ and a measure $\rho$ on $(0,\infty)$ such that $\int_{(0,\infty)} \min\{t,1\}\,\rho(dt)<\infty$. Typical examples are the fractional powers $\lambda^s$, $s\in (0,1)$, with $(\beta,\rho(dt))=\left(0, s\Gamma(1-s)^{-1} t^{-1-s}\,dt\right)$ or the logarithm $\log(1+\lambda)$ with $(\beta,\rho(dt))=\left(0, t^{-1}e^{-t}\,dt\right)$.

Every Bernstein function $f$ corresponds to an increasing one-dimensional L\'evy process (a so-called \emph{subordinator}) $(T_t)_{t\geq 0}$ via the Laplace transform
\begin{gather*}
	\Ee e^{-\lambda T_t} = e^{-tf(\lambda)},\quad t,\lambda >0,
\end{gather*} 
and Bochner's subordination principle tells us that $-f(-A_R)$ is the infinitesimal generator of the time-changed L\'evy process $R_{N_{T_t}}$, $t\geq 0$. The symbol of the infinitesimal generator is $-f(\psi_R(\xi))$ with $\psi_R$ from \eqref{dis-e01}.

We will need the following lemma on extensions of a Bernstein function \enquote{to the left}.
\begin{lemma}\label{dis-03}
	Let $f:(0,\infty)\to [0,\infty)$ be a Bernstein function given by \eqref{dis-e06}, let $\alpha>0$, and let $(T_t)_{t\geq 0}$ be the corresponding subordinator. The following assertions are equivalent.
	\begin{enumerate}\itemsep=6pt
	\item[\textup a)]
		$\Ee e^{\alpha T_t}<\infty$ for some \textup{(}hence, all\textup{)} $t>0$. 
	\item[\textup b)]
		$\int_{(1,\infty)} e^{\alpha s}\,\rho(ds)<\infty$. 
	\item[\textup c)]
		$f$ has an extension onto $(-\alpha,\infty)$, which is continuous in $[-\alpha,\infty)$, infinitely often differentiable in $(-\alpha,\infty)$ such that $(-1)^{k+1} f^{(k)}\geq 0$ for all $k\in\nat$.
	\item[\textup d)]
		$f$ has an real-analytic extension onto $(-\alpha,\infty)$, which is continuous in $[-\alpha,\infty)$.
	\end{enumerate}
\end{lemma}
\begin{proof}
	The equivalence a)$\Leftrightarrow$b) is a classic result from probability theory, see e.g.\ \cite[Thm.~25.3]{Sato}. 

\bigskip
	b)$\Rightarrow$c): If b) holds, it is clear that the integral representation \eqref{dis-e06} extends continuously to all $x\in [-\alpha,\infty)$. Moreover, we can use the differentiation lemma for parameter-dependent integrals to see that all derivatives exist and satisfy $(-1)^{k+1}f^{(k)}\geq 0$. Notice, that $f|_{[\alpha,0)}\leq 0$ while $f|_{(0,\infty)}\geq 0$. This proves c).

\bigskip	
	c)$\Rightarrow$b): If c) holds, the function $g(x) := f(x-\alpha)-f(-\alpha)$ is a Bernstein function with $g(0)=0$, and as such it has a unique representation
	\begin{gather*}
		f(x-\alpha)-f(-\alpha) = \beta_\alpha x + \int_{(0,\infty)} \left(1-e^{-xt}\right)\rho_\alpha(dt),\quad
		\beta_\alpha>0,\; \int_{(0,\infty)} \min\{1,x\}\,\rho_\alpha(dt)<\infty.
	\end{gather*}
	On the other hand, from the standard representation \eqref{dis-e06} of the Bernstein function $f$ we get
	\begin{gather*}
		f(x-\alpha) = \beta x + \int_{(0,\infty)} \left(1-e^{-(x-\alpha)t}\right)\rho(dt),\quad x>\alpha.
	\end{gather*} 
	Comparing these representations, we have
	\begin{gather*}
		\beta_\alpha 
		= \lim_{x\to\infty} \frac{f(x-\alpha)-f(-\alpha)}{x}
		= \lim_{x\to\infty} \frac{f(x-\alpha)}{x}
		= \beta.
	\end{gather*}
	Moreover, from the first formula we get $f(-\alpha)= \int_{(0,\infty)} \left(1-e^{-\alpha t}\right)\rho_\alpha(dt)$, and so
	\begin{gather*}
		\int_{(0,\infty)} \left(e^{-t\alpha}-e^{-tx}\right) \rho_\alpha(dt) = \int_{(0,\infty)} \left(1-e^{-t(x-\alpha)}\right) \rho(dt),
		\quad x>\alpha.
	\end{gather*}
	Because of the uniqueness of the representing measure $\rho$, we conclude that $e^{t\alpha}\,\rho(dt) = \rho_\alpha(dt)$, i.e.\ $\int_{(1,\infty)} e^{t\alpha}\,\rho(dt) = \rho_\alpha(1,\infty)<\infty$. This proves b).

\bigskip	
	b)$\Rightarrow$d): If b) holds, it is easy to see that the representation \eqref{dis-e06} extends to all $\zeta = \lambda + i\mu$ with $\lambda > -\alpha$ and $\mu\in\real$. Just observe that
	\begin{gather*}
		\left|1-e^{-t(\lambda+i\mu)}\right|
		\leq t\sqrt{\lambda^2+\mu^2}\I_{(0,1)}(t) + \left(1+e^{\alpha}\right)\I_{[1,\infty)}(t).
	\end{gather*}
	This extension is analytic on the strip $(-\alpha,\infty) + i(-1,1)$ and continuous up to the boundary, hence $f|_{(\alpha,\infty)}$ is real-analytic and $f(\alpha+)$ exists. This proves d).

\bigskip
	d)$\Rightarrow$c): Finally, assume that d) holds. For some interval $(-\delta,\delta)$ the power series $f(x) = \sum_{k=0}^\infty \frac 1{k!} f^{(k)}(0) x^k$ converges, and we get for the derivatives
	\begin{gather*}
		(-1)^{m+1}f^{(m)}(x) 
		= \sum_{k=m}^\infty \frac{f^{(k)}(0)}{k!} \frac{k!}{m!} (-1)^{-m-1}x^{k-m}
		= \sum_{k=m}^\infty \frac{(-1)^{k+1}f^{(k)}(0)}{m!} (-x)^{k-m}.
	\end{gather*}
	This expression is positive for all $x\in (-\delta,0)$, since $(-1)^{k+1} f^{(k)}(0)\geq 0$. Using a standard chaining argument shows that every compact interval $[-\alpha + l^{-1},0]$, $l\in\nat$, can be covered by finitely many intervals of the type $(x_i-\delta_i, x_i+\delta_i)$ such that $(-1)^{m+1}f^{(m)}(x_i)\geq 0$. From the Taylor series around $x_i$ it follows that $(-1)^{m+1}f^{(m)}(x)\geq 0$  for all $x\in (x_i-\delta_i,x_i)$. Thus, $(-1)^{m+1} f^{(m)}(x)\geq 0$ holds for all $x\in [-\alpha+l^{-1},0)$, hence for all $x\in (-\alpha,0)$, and c) follows.
\end{proof}

We will now prove the UCP for a class of discrete L\'evy operators.
\begin{theorem}
	Let $(T_t)_{t\geq 0}$ be a subordinator with Bernstein function $f$ such that $(T_t)_{t\geq 0}$ has no strictly positive exponential moments, i.e.\ $\Ee e^{\alpha T_t}=\infty$ for any $\alpha>0$. Let $(R_m)_{m\in\nat}$ be a random walk on $\integer^n$ and denote by $(R_{N_{T_t}})_{t\geq 0}$ the L\'evy process with characteristic exponent $\psi(\xi) = f(\psi_{R}(\xi))$ and generator $-\psi(D) = -f(-A_R)$. Let $h>1$ and $(u_k)_{k\in \integer^n}$ be a sequence such that
	\begin{gather*}
		\forall k\in\integer^n\setminus B_h(0)\::\:
		(\psi(D)u)_k = u_k = 0.
	\end{gather*}
	Then $u_k=0$ for all $k\in\integer^n$.
\end{theorem}
Theorem applies, in particular, to fractional powers of the discrete Laplacian: Take the symmetric, simple random walk $R$ on $\integer^n$ and observe that $f(x)=x^s$, $s\in (0,1)$, cannot be extended beyond $x=0$ as a real-analytic function, cf.\ Lemma~\ref{dis-03}
\begin{proof}
	The proof follows closely the idea of the proof of \cite[Theorem 3]{ber-et-al}. Assume that $u\neq 0$. Let $\left(S_{N_t}\right)_{t\geq 0}$ be any non-degenerate random walk on a lattice $L\subset\real$ embedded in continuous time, and $\alpha\in \real\setminus\{0\}$. From \eqref{dis-e01} we know that
	\begin{gather}\label{dis-e10}\begin{aligned}
		\psi_{S}(-i\alpha)+\psi_{S}(i\alpha)
		&= \sum_{\ell\in L}p_\ell\left(1-e^{\alpha \ell}\right) + \sum_{\ell\in L}p_\ell\left(1-e^{-\alpha \ell}\right)\\
		&= \sum_{\ell\in L}p_\ell\left(2-e^{\alpha \ell}+e^{-\alpha \ell}\right) < 0,
	\end{aligned}\end{gather}
	i.e.\ $\psi_{S}(-i\alpha)<0$ or  $\psi_{S}(i\alpha)<0$. Since $R$ is non-degenerate, there is some $z_0\in\complex^n$ such that $S_m := z_0\cdot R_m$ is a non-degenerate random walk on $L = z_0\cdot\integer^n\subset\real$.
	
	Using Tonelli's theorem, we get
	\begin{align*}
		\Ee \left(e^{\alpha z_0\cdot R_{N_{T_t}}} + e^{-\alpha z_0 R_{N_{T_t}}}\right)
		&= \int_{(0,\infty)} \Ee \left(e^{\alpha z_0\cdot R_{N_{r}}} + e^{-\alpha z_0\cdot R_{N_{r}}}\right) \Pp(T_t\in dr)\\
		&= \int_{(0,\infty)} \left(e^{-r\psi_S(-i\alpha)} + e^{-r\psi_S(i\alpha)}\right) \Pp(T_t\in dr) = \infty,
	\end{align*}
	since $T_t$ has no strictly positive exponential moments.
	
	Now take the Fourier transform (Fourier series) of $u$ and $\psi(D)u$, which we denote by $\Fcal u$ and $\Fcal(\psi(D)u)$. We know that
	\begin{align*}
		\Fcal(\psi(D)u)(z) = \psi(z)\Fcal u(z).
	\end{align*}
	As $u$ and $\psi(D)u$ are compactly supported, we conclude that $\Fcal u$ and $\Fcal(\psi(D)u)$ are entire functions on $\complex^n$. 
	
	There has to exist a directional vector $z_0\in\complex^n$ such that $\Fcal u(\cdot z_0):\complex\to\complex$ and $\psi(\cdot z_0):\complex\to\complex$ are holomorphic and not identically $0$, as $\psi(\cdot z_0)$ can only be equal to $0$ in a $n-1$ dimensional vector space, otherwise $(R_{N_{T_t}})_{t\geq 0}$ would be degenerate, see the proof of \cite[Lemma~4.3]{Berger-Schilling}. 
	
	So we see that
	\begin{align*}
		t\mapsto\psi(t z_0) = \frac{\Fcal(\psi(D)u)(tz_0)}{\Fcal u(tz_0)},
	\end{align*} 
	is meromorphic; in particular $\psi(\pm i\delta z_0)$ exists for some $\delta>0$. In view of \eqref{dis-e01}, see also the related calculation \eqref{dis-e10}, we get $\psi_{R}(i\delta z_0)+\psi_{R}(-i\delta z_0)<0$.
	
	Using the fact that $T_t$ has no positive exponential moments, we know from Lemma~\ref{dis-03} that $f$ cannot be extended into the left half-axis. Since $\psi(\pm i\delta z_0) = f(\psi_{R}(\pm i\delta z_0))$ we get a contradiction, unless $u=0$.
\end{proof}
The next theorem shows that the UCP property does not hold for bounded sets:
\begin{theorem}
	Let $(T_t)_{t\geq 0}$ be a subordinator with Bernstein function $f$ and $(R_m)_{m\in\nat}$ be a random walk on $\integer^n$. Denote by $(R_{N_{T_t}})_{t\geq 0}$ the L\'evy process with characteristic exponent $\psi(\xi) = f(\psi_{R}(\xi))$ and generator $-\psi(D) = -f(-A_R)$. If the support of the L\'evy measure $\nu$ of $R_{N_{T_t}}$ has infinitely many points, then for every $h>0$ there always exist a sequence $(u_k)_{k\in \integer}\neq 0$ such that
	\begin{gather*} 
		\forall k\in B_h(0)\::\:
		u_k=0\quad\text{and}\quad
		(\psi(D)u)_k.
	\end{gather*}
\end{theorem}
\begin{proof}
	The equation
	\begin{align*}
		(\psi(D)u)_k = 0\text{\ \ for\ \ }k\in B_h(0)
	\end{align*}
	can be seen as an undetermined systems of linear equations, as it depends on infinitely many $u_k$'s by our assumption that $\#\, \supp(\nu)=\infty$. We can now solve this system of equations under the condition that $u_k=0$ for $k\in B_h(0)$.
\end{proof}

\end{document}